\documentclass[12pt,a4paper,twoside]{article}
\usepackage{a4wide,amsmath,amsbsy,amsfonts,amssymb,stmaryrd,amsthm,mathrsfs}
\newcommand\ZZ{{\hat{\mathbb Z}}}
\newcommand\Z{{\mathbb Z}}

\newcommand\Q{{\mathbb Q}}
\newcommand\D{{\mathbb D}}
\newcommand\mS{{\mathbb S}}

\newcommand\R{{\mathbb R}}

\newcommand\N{{\mathbb N}}
\newcommand\cL{{\mathscr L}}

\newcommand\ra{\rightarrow}

\newcommand\ilim{\varprojlim}

\newcommand\Spec{\operatorname{Spec}}
\newcommand\aut{\operatorname{Aut}}
\newcommand\out{\operatorname{Out}}
\newcommand\inn{\operatorname{inn}}

\newcommand\lra{\longrightarrow}
\newcommand\hookra{\hookrightarrow}

\newcommand\da{\downarrow}

\renewcommand{\hom}{\mathrm{Hom}}

\renewcommand{\Im}{\mathrm{Im}}
\newcommand\sr{\stackrel}
\newcommand\st{\scriptstyle}
\newcommand\sst{\scriptscriptstyle}

\newcommand\hP{\widehat{\Pi}}

\newcommand\ssm{\smallsetminus}
\newcommand\ol{\overline}

\newcommand\cC{{\mathscr C}}

\newcommand\Ld{\Lambda}
\newcommand\wh{\widehat}

\newcommand\td{\tilde}

\newcommand\gm{\gamma}

\def\co{\colon\thinspace}

\newtheorem{theorem}{Theorem}[section]
\newtheorem{corollary}[theorem]{Corollary}
\newtheorem{proposition}[theorem]{Proposition}
\newtheorem{lemma}[theorem]{Lemma}

\theoremstyle{definition}
\newtheorem{definition}[theorem]{Definition}
\newtheorem{remark}[theorem]{Remark}
\newtheorem{remarks}[theorem]{Remarks}

\begin{document}

\title{Characterizing closed curves on Riemann surfaces\\ via homology groups of coverings}
\author{Marco Boggi \and Pavel Zalesskii}\maketitle

\begin{abstract}
Let $S$ be a hyperbolic oriented Riemann surface of finite type. The main purpose of this paper is
to show that non-trivial geometric intersection between closed curves on $S$ is detected by some
symplectic submodules they naturally determine in the homology groups of the compactifications of
unramified $p$-coverings of $S$, for $p\geq 2$ a fixed prime. In particular, this gives a
characterization of simple closed curves on $S$ in terms of homology groups of $p$-coverings.

In Section~\ref{group theoretic}, we define a $p$-adic Reidemeister pairing on the fundamental
group of $S$ and show that the free homotopy classes of two loops have trivial geometric intersection 
if and only if they are orthogonal with respect to this pairing.

As an application, we give a geometric argument to prove that oriented surface groups are conjugacy
$p$-separable (a combinatorial proof of this fact was recentely given by Paris \cite{Paris}).
\newline

\end{abstract}

\section{Introduction}
In the papers \cite{Jaco} and \cite{Sta}, Stallings and Jaco established the equivalence between
the Poincar\'e conjecture, now a celebrated theorem by Perelman, and the following
group-theoretical statement:
\begin{itemize}
\item[$\star$]{\it Let $S_g$ be a closed oriented surface of genus $g\geq 2$. Let $F_g$ be a free
group of rank $g$. Let $\eta\co\pi_1(S_g,s_0)\ra F_g\times F_g$ be an epimorphism. Then, there
is a non-trivial element in the kernel of $\eta$ which may be represented by a simple closed curve in $S_g$.}
\end{itemize}
Of course, it is still an interesting problem to provide a group-theoretic proof of
the above statement. The first step in this direction is to give an algebraic characterization
of simple closed curves on the closed Riemann surface $S_g$.

A program in this sense was formulated by Turaev \cite{Tu}. Progress in this direction have
been recently accomplished by Chas, Gadgil and Krongold (see  \cite{C-G}, \cite{C-K},
\cite{C-K2}), who characterize simple closed curves on a Riemann surface
in terms of the Goldman Lie algebra.

In Section~\ref{simple}, we give an elementary criterion to characterize simple closed curves on
any hyperbolic Riemann surface in terms of the intersection pairing on the closures
of normal unramified $p$-coverings of the given Riemann surface, for a fixed prime $p$.
The proof is based on the hyperbolic geometry of the $p$-adic solenoid,
developed by means of some elementary pro-$p$ group theory.

In Section~\ref{group theoretic}, we refine further this characterization by means of a $p$-adic
Reidemeister pairing. This is a quite natural pro-$p$ version of the classical one (cf.\ Section~3 \cite{Hempel2}).
We show that, an element $\gm\in\pi_1(S_g,s_0)$ contains a simple closed curve in its free homotopy class
if and only if it is singular for the $p$-adic Reidemeister pairing. In particular, whether the
free homotopy class of $\gm$ contains or not an embedded representative can be determined purely in
terms of the group structure of the fundamental group $\pi_1(S_g,s_0)$. 

In Section~\ref{separating}, we apply the above characterization of simple closed curves to show
that it is possible to distinguish homotopy classes of closed curves on a hyperbolic Riemann
surface in terms of the first homology groups of its normal unramified $p$-coverings. An easy
consequence is conjugacy $p$-separability of surface groups.

\section{Characterization of simple closed curves}\label{simple}
Let $S_g$ be a compact oriented Riemann surface without boundary of genus $g$ and let
$S_{g,n}:=S_g\ssm\{P_1,\ldots,P_n\}$ be the same surface from which $n$ distinct points have
been removed. We assume that $\chi(S_{g,n})=2-2g-n<0$. Even if most of the stated results
hold also in the non-orientable case, for simplicity, we only consider the orientable case.

\emph{A closed curve on $S_{g,n}$} is a continuous map from the circle $S^1$ to
the surface $S_{g,n}$. For a fixed base point $a$, let us denote by $\Pi_{g,n}$ the fundamental
group of $S_{g,n}$. Then, it is well known that the set of homotopy classes of closed curves on
$S_{g,n}$ identifies with the quotient of the sets $\{\{\alpha,\alpha^{-1}\}|\,\alpha\in\Pi_{g,n}\}$ and
$\{\langle\alpha\rangle |\,\alpha\in\Pi_{g,n}\}$ by the action induced by inner automorphisms
of $\Pi_{g,n}$. 

Given a covering map $p\co S'\ra S_{g,n}$ and a closed curve $\gm\co S^1\ra S_{g,n}$,
the connected components of the fiber product $S^1\times_{S_{g,n}}S'$ are homeomorphic to $S^1$.
Let us then call the restriction $\td{\gm}\co S^1\ra S'$ to a connected component of the pull-back map
$\gm'\co S^1\times_{S_{g,n}}S'\ra S'$  \emph{an elevation of $\gm$ to $S'$} (this terminology is borrowed
from \cite{FH}).

\begin{definition}\label{submodule}For $K$ a finite index subgroup of $\Pi_{g,n}$, let
$p_K\co S_K\ra S_{g,n}$ be the associated covering, $\ol{S}_K$ the compact Riemann surface
obtained filling in the punctures of $S_K$ and $H_1(\ol{S}_K)$ the first homology group
of $\ol{S}_K$ with $\Z$-coefficients. Let us then define $V^K_\gm$ to be the submodule of
$H_1(\ol{S}_K)$ generated by the cycles supported on the elevations of $\gm$ to $S_K$.
\end{definition}

For a normal finite index subgroup $K$ of $\Pi_{g,n}$, let $G_K$ be the deck
transformation group of $p_K\co S_K\ra S_{g,n}$. Given a closed curve $\gm$ in $S_{g,n}$,
the group $G_K$ then acts naturally and transitively on the set of elevations of
$\gm$ to $S_K$ and $V^K_\gm$ is a $G_K$-invariant submodule of $H_1(\ol{S}_K)$.

For a closed curve $\gm$ on $S_{g,n}$, let $\td{\gm}$ be an element of $\Pi_{g,n}$ whose free
homotopy class contains the closed curve $\gm$ and let $k>0$ be the smallest integer such that
$\td{\gm}^k\in K$. Then, the submodule $V^K_\gm$ can also be characterized as the image of
the normal subgroup $\langle\td{\gm}^k\rangle^{\Pi_{g,n}}\cap K$ in the homology group
$H_1(\ol{S}_K)$, where we denote by $\langle\td{\gm}^k\rangle^{\Pi_{g,n}}$ the smallest 
normal subgroup of $\Pi_{g,n}$ which contains $\td{\gm}^k$. Let us also define $V^K_{\td{\gm}}:=V^K_\gm$.

For an integer $s>0$, \emph{an $s$-power} of a closed curve $\gm\co S^1\ra S_{g,n}$ is a closed
curve $\gm^s\co S^1\ra S_{g,n}$ which factors through a continuous map $S^1\ra S^1$ of degree $s$
and $\gm$. A closed curve $\gm$ on the Riemann surface $S_{g,n}$ is called \emph{non-power} or
\emph{primitive} if it is not homotopic to an $s$-power of a closed curve, for $s>1$.
Let us observe that an elevation of a non-power closed curve is also non-power.

\emph{A simple closed curve (briefly s.c.c.) on $S_{g,n}$} is an embedded circle $S^1\hookra S_{g,n}$.
An interesting problem is that of establishing when a non-power closed curve $\gm$ on
$S_{g,n}$ has in its homotopy class an embedded representative. In this section, we will give a
characterization of this property in terms of the homology of finite unramified coverings of $S_{g,n}$.

More generally, we will be able to determine in this way when the geometric intersection number
$|\gm\cap\gm'|_G$ between two closed curves $\gm$ and $\gm'$ is zero or one.

In order to formulate the main result, it is convenient to introduce some natural topologies on
the fundamental group $\Pi_{g,n}$.

\emph{A profinite topology} on a group $G$ is a topology (making it a topological group) for which
a  neighborhood basis  of the identity consists of finite index subgroups.
The usual way to define such a topology on the group $G$ is to fix a class of
finite groups $\cC$ in the sense of the definition below:

\begin{definition}\emph{A class of finite groups} (cf. Definition~3.1 in \cite{A-M}) is a full 
subcategory $\cC$ of the category of finite groups which is closed under taking subgroups, 
homomorphic images and extensions (meaning that a short exact sequence of finite groups is 
in $\cC$ whenever its exterior terms are). We always assume that $\cC$ contains a nontrivial group. 
\end{definition}

Since $\cC$ contains a nontrivial group, it will contain $\Z/p$ for some prime $p\ge 2$ and then it
easily follows that $\cC$ contains all the finite groups whose order is a power of $p$.

The finite groups whose order is a power of $p$ are nilpotent and form a class of finite groups by
themselves (denoted sometimes $(p)$). So these provide the minimal examples.

Given a group $G$, then we say that a subgroup $H\leq G$ is \emph{$\cC$-open} if it contains a
normal subgroup $N$ of $G$ such that the quotient group $G/N$ belongs to $\cC$. In this case, we
write $H\leq_\cC G$ and $N\lhd_\cC G$, respectively. A fundamental system of neighborhoods of the identity
for the pro-$\cC$ topology on the group $G$ is given by any cofinal system of
$\cC$-open normal subgroups. We say that a subgroup $H$ of $G$ is $\cC$-closed if $H$ is
a closed subset in the pro-$\cC$ topology of $G$ and that a map of groups $G\ra G'$ is
$\cC$-continuous if it is continuous in the pro-$\cC$ topologies of $G$ and $G'$.
Note that a monomorphism is always $\cC$-continuous.

\emph{The pro-$\cC$ completion} of a group $G$ is defined to be the inverse limit:
$$\widehat{G}^\cC:=\varprojlim_{\sst N\lhd_{\cC} G} G/N.$$
Let us endow the finite groups $G/N$ with the discrete topology and the group $\wh{G}^\cC$ with
the subspace topology induced by the natural monomorphism
$\wh{G}^\cC\hookra\prod_{\sst N\lhd_{\cC} G} G/N$. Then, the profinite group $\wh{G}^\cC$
is a compact, Hausdorff, totally disconnected, topological group and a fundamental system of
neighborhoods of the identity is provided by the kernels of the natural epimorphisms
$\wh{G}^\cC\ra G/N$, for $N\lhd_\cC G$. There is a natural homomorphism of groups
$G\ra\wh{G}^\cC$ with dense image and the pro-$\cC$ topology on $G$ is, by definition,
the weakest topology for which this map is continuous.

With the above notations and definitions, let us now state the main result of the paper:

\begin{theorem}\label{geo int}Let $\cC$ be a class of finite groups. Then, a pair of (not necessarily distinct) 
closed curves $\gm$ and $\gm'$ on a Riemann surface $S_{g,n}$ have trivial geometric intersection
if and only if, for a cofinal system of $\cC$-open subgroups $\{K\}$ of $\Pi_{g,n}$, there holds
$\langle x,y\rangle_K=0$, for all $x\in V^K_\gm$ and all $y\in V^K_{\gm'}$, where
$\langle\_ ,\_\rangle_K$ is the intersection pairing on the first integral homology group
of the closed Riemann surface $\ol{S}_K$.
 \end{theorem}

Taking $\gm'=\gm$ in Theorem~\ref{geo int}, we get the aforementioned characterization of s.c.c.'s:

\begin{corollary}\label{s.c.c.1}The homotopy class of a non-power closed curve $\gm$ on $S_{g,n}$ contains a
simple closed curve if and only if, for a cofinal system of
$\cC$-open subgroups $\{K\}$ of $\Pi_{g,n}$, the associated submodules
$V^K_\gm$ of $H_1(\ol{S}_K)$ are totally isotropic.
\end{corollary}

Thanks to Corollary~4.4 in \cite{Hempel2}, we can also give an algebraic criterion to decide whether
two non-homotopic s.c.c.'s $\alpha$ and $\beta$ on $S_{g,n}$ have geometric intersection one. 

\begin{corollary}\label{Hempel}Let $\alpha$ and $\beta$ be non-homotopic s.c.c.'s on $S_{g,n}$.
Then $\alpha$ and $\beta$ are homotopic to s.c.c.'s which meet transversally in a single point
if and only, for some $\td{\alpha},\td{\beta}\in\Pi_{g,n}$, whose free homotopy classes contain,
respectively, $\alpha$ and $\beta$, and for a cofinal system of $\cC$-open subgroups $\{K\}$ of $\Pi_{g,n}$, 
the associated submodules $V^K_{[\td{\alpha},\td{\beta}]}$ of $H_1(\ol{S}_K)$ are totally isotropic.
\end{corollary}

The case in which Theorem~\ref{geo int} and its corollaries are more interesting is
undoubtably when $\cC$ is the class of finite $p$-groups. In particular, the applications given in
Section~\ref{separating} make use only of this case.

\section{The pro-$\cC$ hyperbolic solenoid}\label{solenoid section}
Let us fix a complete, non-singular metric on $S_{g,n}$ of constant curvature $-1$.
A closed curve $\gm$ on $S_{g,n}$ is \emph{peripheral} if it is homotopic to a
power of a s.c.c.\ bounding a $1$-punctured disc on $S_{g,n}$. Peripheral closed curves are
characterized by the property that they have trivial geometric intersection with any closed curve
on the surface. With the chosen metric, all non-peripheral closed curves on $S_{g,n}$ have a
unique geodesic representative in their homotopy class.

Let $\D\ra S_{g,n}$ be the universal covering space. The choice of a point
$\td{a}\in\D$, lying above the base point $a\in S_{g,n}$, identifies the fundamental group
$\Pi_{g,n}$ with the covering transformation group of $\D\ra S_{g,n}$.
The space $\D$, with the induced metric, can then be identified with the Poincar\'e disc and the
fundamental group $\Pi_{g,n}$ with a discrete subgroup of
$\mathrm{Aut}(\D)\cong\mathrm{PSL}_2(\R)$. Note that all elements of $\Pi_{g,n}$ are identified
with hyperbolic elements of $\mathrm{PSL}_2(\R)$, except those containing a peripheral curve
in their homotopy class, which are instead identified with parabolic elements.

For $\cC$ a given class of finite groups, let $\hP_{g,n}^\cC$ be the pro-$\cC$ completion of the
fundamental group $\Pi_{g,n}$.
Since $\cC$ contains the class of finite $p$-groups, for some prime $p>1$, by Theorem~\ref{injective},
there is a natural monomorphism $\Pi_{g,n}\hookra\hP_{g,n}^\cC$. By means of this monomorphism,
let us identify $\Pi_{g,n}$ with its image in the pro-$\cC$ group $\hP_{g,n}^\cC$.

\emph{The pro-$\cC$ hyperbolic solenoid} $\mS_{g,n}^\cC$ is defined to be the inverse limit space:
$$\mS_{g,n}^\cC:=\varprojlim_{\sst K\lhd_\cC\Pi_{g,n}}S_K\cong
\varprojlim_{\sst K\lhd_\cC\Pi_{g,n}}\left.[\D\times(\Pi_{g,n}/K)]
\right/\Pi_{g,n}=\left.\D\times\hP_{g,n}^\cC\right/\Pi_{g,n},$$
where an element $x\in\Pi_{g,n}$ acts on the product $\D\times\hP_{g,n}^\cC$ by the formula
$x\cdot(d,\beta)=(x\cdot d,x\beta)$. This is a generalization of the hyperbolic
$n$-punctured, genus $g$ solenoid. For $\cC=(p)$, the pro-$p$ solenoid
$\mS_{g,n}^{(p)}$ is called \emph{the $p$-adic solenoid}. For $\cC$ the class of all finite groups,
the pro-$\cC$ solenoid is denoted simply by $\mS_{g,n}$ and called \emph{the solenoid}.

From the above realization of the pro-$\cC$ solenoid and the fact that the natural homomorphism
$\Pi_{g,n}\ra\hP_{g,n}^\cC$ is injective, it follows that the inverse limit of the
natural covering maps $\D\ra\D/K$ is a $\Pi_{g,n}$-equivariant embedding $\D\hookra\mS_{g,n}^\cC$
with dense image. Moreover, under the isomorphism described above, the image of the
Poincar\'e disc in the solenoid identifies with the image of $\D\times\Pi_{g,n}$ in the quotient
$\D\times\hP_{g,n}^\cC/\Pi_{g,n}$. Let us identify $\D$ with its image in the $p$-adic solenoid.
In this way, we get \emph{the preferred analytic leaf $\D$} of the solenoid. The other leaves are
obtained translating this one by the action of $\hP_{g,n}^\cC$. 

From the above construction, it is clear that the preferred leaf is determined by the canonical monomorphism
$\Pi_{g,n}\hookra\hP_{g,n}^\cC$ (a discretification of $\hP_{g,n}^\cC$). Twisting this canonical
monomorphism by the inner automorphism
$\inn_\gm\co x\mapsto\gm x\gm^{-1}$, for $\gm\in\hP_{g,n}^\cC$, we get another realization of the
pro-$\cC$ solenoid whose preferred leaf is instead the translated leaf $\gm\cdot\D$:
$$\mS_{g,n}^\cC\cong\left.\gm\cdot\D\times\hP_{g,n}^\cC\right/\Pi_{g,n},$$
where now an element $x\in\Pi_{g,n}$ acts on the product $\gm\cdot\D\times\hP_{g,n}^\cC$ by the
formula $x\cdot(d,\beta)=(\inn_\gm(x)\cdot d,\inn_\gm(x)\beta)$.

\begin{proposition}\label{solenoid}The pro-$\cC$ solenoid $\mS_{g,n}^\cC$ is a connected
Hausdorff topological space endowed with a natural continuous action of the pro-$\cC$ group
$\hP_{g,n}^\cC$, with quotient the Riemann surface $S_{g,n}$.
\end{proposition}

\begin{proof}The pro-$\cC$ solenoid contains the disc $\D$ as a dense connected subspace and a topological
space containing a connected dense subset is connected.
\end{proof}

If $\Ld_\cC$ is the set of primes which occur as orders of groups in $\cC$, then
there holds $\ZZ^\cC\cong\prod_{p\in\Ld_\cC}\Z_p$. For every prime $p\in\Ld_\cC$, there
is a natural epimorphism $\hP_{g,n}^\cC\ra\hP_{g,n}^{(p)}$. It follows that the pro-cyclic subgroup
topologically generated by an element $\gm\in\Pi_{g,n}$ inside the pro-$\cC$ completion $\hP_{g,n}^\cC$
is isomorphic to $\ZZ^\cC$. Let us denote this group simply by $\gm^{\ZZ^\cC}$.

Let us describe the closure inside the pro-$\cC$ solenoid $\mS_{g,n}^\cC$ of a hyperbolic line
$\ell$ of $\D$ (also called \emph{a geodesic}), obtained as the inverse image of a closed geodesic
curve in $S_{g,n}$:

\begin{proposition}\label{closure geo}Let $\ell$ be a geodesic on the Poincar\'e disc $\D$ whose
image in $S_{g,n}$ is a closed curve. Let us then identify $\ell$ with its canonical image in the
pro-$\cC$ solenoid $\mS_{g,n}^\cC$. The closure $\ol{\ell}$ of this geodesic inside $\mS_{g,n}^\cC$
is naturally isomorphic to the one-dimensional pro-$\cC$ solenoid $\R\times\ZZ^\cC/\Z$, where
$z\in\Z$ acts on $\R\times\ZZ^\cC$ by the formula $z\cdot(r,s)=(z+ r,z+s)$. In particular, the closure
$\ol{\ell}$ has no self-intersection points.
\end{proposition}

\begin{proof}There is a non-power hyperbolic element $\gm\in\Pi_{g,n}$ whose action on $\D$
has for axis the given geodesic $\ell$. Let then $\gm^\Z\cong\Z$ be the cyclic
subgroup generated by $\gm$ in $\Pi_{g,n}$.

The geodesic $\ell$ identifies with the subspace $\ell\times\gm^\Z/\gm^\Z$ of the solenoid
$\D\times\hP_{g,n}^\cC/\Pi_{g,n}$. Let us then show that its closure there identifies with the image
of the closed subspace $\ell\times\gm^{\ZZ^\cC}$ of $\D\times\hP_{g,n}^\cC$ in the pro-$\cC$
solenoid and that this image is isomorphic to the quotient of the closed subspace
$\ell\times\gm^{\ZZ^\cC}$ by its stabilizer $\gm^{\ZZ^\cC}\cap\Pi_{g,n}$,
for the action of the group $\Pi_{g,n}$ on the space $\D\times\hP_{g,n}^\cC$ described above.
There is at least a natural continuous map:
$$\varphi_\gm\co\left.\ell\times\gm^{\ZZ^\cC}\right/\gm^\Z
\ra\left.\D\times\hP_{g,n}^\cC\right/\Pi_{g,n}\equiv\mS_{g,n}^\cC.$$

Let us observe that, since the element $\gm\in\Pi_{g,n}$ is non-power, the cyclic subgroup $\gm^\Z$
is self-centralizing in $\Pi_{g,n}$. Therefore, there holds the identity
$\gm^{\ZZ^\cC}\cap\Pi_{g,n}=\gm^\Z$,
otherwise stated, the subgroup $\gm^\Z$ is closed in the pro-$\cC$ topology
of $\Pi_{g,n}$. It follows that the map $\varphi_\gm$ is injective since two points $(d,\beta)$ and
$(d',\beta')$ of $\ell\times\gm^{\ZZ^\cC}$ are mapped to the same point of the $p$-adic solenoid if
and only if there is an $x\in\Pi_{g,n}$ such that $(x\cdot d,x\beta)=(d',\beta')$
and this happens only if $x\in\gm^{\ZZ^\cC}\cap\Pi_{g,n}=\gm^\Z$.

Now, the quotient space $\ell\times\gm^{\ZZ^\cC}/\gm^\Z$ is homeomorphic to
the quotient space $\R\times\ZZ^\cC/\Z$, described in the statement of the proposition, and this is
a compact Hausdorff topological space. Therefore, the map $\varphi_\gm$ is a homeomorphism onto its image
$\Im\,\varphi_\gm$ which then is a closed subset of the pro-$\cC$ solenoid, because this is a
Hausdorff space. Since the image of $\varphi_\gm$ contains the geodesic $\ell$ as a dense subset,
it follows that $\Im\,\varphi_\gm$ is actually the closure of $\ell$ inside the pro-$\cC$ solenoid
$\mS_{g,n}^\cC$.
\end{proof}

We can now improve the cyclic case of a classical result by Scott \cite{Scott1}:

\begin{theorem}\label{closed curve}Let $\gm$ be a non-power, non-peripheral closed curve on the
Riemann surface $S_{g,n}$. Then, for any fixed prime $p$, there is a normal, unramified $p$-covering
$p_L\co S_L\ra S_{g,n}$ such that every elevation of $\gm$ to $S_L$ is
homotopic to a non-separating simple closed curve.
\end{theorem}

\begin{proof}Let us assume that the closed curve $\gm$ is immersed in $S_{g,n}$ and is the
geodesic representative in its homotopy class.
Let $\ell$ be a lift of $\gm$ to $\D$ which is then a hyperbolic line. By Proposition~\ref{closure geo},
the closure $\ol{\ell}$ of this geodesic inside the $p$-adic solenoid $\mS_{g,n}^{(p)}$ is naturally
isomorphic to $\R\times\Z_p/\Z$ and has no self-intersection points.

For a given finite index normal subgroup $K$ of $\Pi_{g,n}$, let us denote by $\gm_K$
the image of $\ell$ in $S_K$. The inverse limit $\ilim_{\sst K\lhd_p\Pi_{g,n}}\gm_K$ is then
naturally identified with the closure $\ol{\ell}$ of the geodesic $\ell$ in the $p$-adic solenoid
$\mS_{g,n}^{(p)}$ and hence it has no self-intersection points.

For a geodesic $\delta$ on a Riemann surface or in the $p$-adic solenoid $\mS_{g,n}^{(p)}$, 
let $\delta\cap_{s}\delta$ be the set of its transversal self-intersection points. There holds then
$$\varprojlim_{\sst K\lhd_p\Pi_{g,n}}(\gm_K\cap_{s}\gm_K)=
\varprojlim_{\sst K\lhd_p\Pi_{g,n}}\gm_K\cap_{s}\varprojlim_{\sst K\lhd_p\Pi_{g,n}}\gm_K=
\ol{\ell}\cap_s\ol{\ell}=\emptyset,$$
where we use the fact that inverse limits commute with inverse limits and, in particular, with intersections.  
Since, for all finite index normal subgroups $K$ of $\Pi_{g,n}$, the set
$\gm_K\cap_{s}\gm_K$ is finite and the inverse limit of a system of non-empty finite sets is
non-empty, we conclude that there is a $p$-open normal subgroup $H$ of $\Pi_{g,n}$
such that there holds $\gm_H\cap_{s}\gm_H=\emptyset$.

By Lemma~3.10 \cite{sym}, there is a characteristic unramified $p$-covering $p_{LH}\co S_L\ra S_H$
such that the connected components of $p_{LH}^{-1}(\alpha)$ are non-separating simple closed curves
for any given non-peripheral s.c.c.\ $\alpha$ on $S_H$. Thus, the induced normal unramified
$p$-covering $p_L\co S_L\ra S_{g,n}$ has the desired properties.
\end{proof}

\begin{remarks}\label{Scott}
\begin{enumerate}
\item With the notations of Theorem~\ref{closed curve}, let $\td{\gm}\in\Pi_{g,n}$ be an
element whose free homotopy class contains $\gm$. In case $\Pi_{g,n}$ is a free group,
i.e.\ for $n\geq 1$, Theorem~\ref{closed curve} follows from
the fact that, by Theorem 5.7 \cite{RZ2}, there is a $p$-open subgroup $H$ of $\Pi_{g,n}$ which
contains the cyclic group $\td{\gm}^\Z$ as a free factor. In the closed surface case instead,
the result is new and can be expressed, group-theoretically, by saying that there is a $p$-open
subgroup $H$ of $\Pi_g$, containing $\td{\gm}$, such that $\td{\gm}$ appears as the
stable letter in a HNN-extension presentation of $H$ (just take $H=\langle L,\td{\gm}\rangle$).

\item The existence of a finite unramified covering with the properties stated in Theorem~\ref{closed curve} 
follows as well from the fact that finitely generated subgroups of a surface group $\Pi_{g,n}$ are 
\emph{geometric}, i.e.\ they can be realized as fundamental groups of 
subsurfaces of finite coverings of $S_{g,n}$ (see \cite{Scott1} and \cite{Scott2}). 
However, in the result of Scott, no condition was imposed on the covering transformation group.
\end{enumerate}
\end{remarks}

In Proposition~\ref{closure geo}, we described the closure in the pro-$\cC$ solenoid of a geodesic
$\ell\subset\D$ which projects to a closed curve in $S_{g,n}$. A much subtler question is to describe
the intersection of the closures of two such geodesics. This is done in the following theorem whose
proof will require a deeper result in the theory of surface groups (Theorem~\ref{intersection}).

\begin{theorem}\label{intersect}Let $\ell$ and $\ell'$ be distinct geodesics of $\D$ which project to
closed curves of $S_{g,n}$ and let, respectively, $\bar\ell$ and $\bar\ell'$ be their closures in the
pro-$\cC$ solenoid $\mS_{g,n}^\cC$. Then, there holds $\bar\ell\cap\bar\ell'=\ell\cap\ell'$, i.e.\ their
intersection is empty when $\ell$ and $\ell'$ are disjoint geodesics of $\D$ and otherwise they meet
in the single point $\ell\cap\ell'$.
\end{theorem}

\begin{proof}For $p\in\Ld_\cC$, there is a normal unramified covering
$\pi_p\co\mS_{g,n}^\cC\ra\mS_{g,n}^{(p)}$ with covering transformation group the kernel of
the natural epimorphism $\hP_{g,n}^\cC\ra\hP_{g,n}^{(p)}$ and the images $\pi_p(\bar\ell)$ and
$\pi_p(\bar\ell')$ are the closures of the geodesics $\ell$ and $\ell'$ in the $p$-adic solenoid.
Therefore, it is enough to prove the theorem for $\cC$ the class of finite $p$-groups,
for $p>1$ a prime.

We can assume that the geodesics $\ell$ and $\ell'$ in $\D$ are the axes, respectively, of two
non-power hyperbolic elements $\td{\gm}$ and $\td{\gm}'$ of $\Pi_{g,n}$, whose free
homotopy classes contain the images of $\ell$ and $\ell'$ in the Riemann surface $S_{g,n}$.

By Proposition~\ref{closure geo}, we then know that the curves $\ol{\ell}$ and $\ol{\ell}'$
in $\mS_{g,n}^{(p)}$ identify with the images of $\ell\times\td{\gm}^{\Z_p}$ and
$\ell'\times\td{\gm}'^{\Z_p}$ in the quotient $\D\times\hP_{g,n}^{(p)}/\Pi_{g,n}\equiv\mS_{g,n}^{(p)}$.

For $s,t\in\Z_p$, the images of two leaves $\ell\times\td{\gm}^s$ and $\ell'\times\td{\gm}'^t$ in the
quotient $\D\times\hP_{g,n}^{(p)}/\Pi_{g,n}$ can possibly intersect only if there exists an
$f\in\Pi_{g,n}$ such that $f\cdot\td{\gm}'^t=\td{\gm}^s$, i.e. only if:
$$f=\td{\gm}^s\td{\gm}'^{-t}\in\Pi_{g,n}\cap\td{\gm}^{\Z_p}\cdot\td{\gm}'^{\Z_p}.$$
If $\td{\gm}'=\td{\gm}^{\pm 1}$, then $\ell=\ell'$, against the hypothesis,
hence $\td{\gm}'\neq\td{\gm}^{\pm 1}$.

For a subgroup $H$ of a group $G$, let us denote by $Z_G(H)$ the centralizer of $H$ in $G$. Since
$\td{\gm}$ and $\td{\gm}'$ are non-power elements of $\Pi_{g,n}$, for all $m\in\Z\ssm\{0\}$, there
hold $Z_{\Pi_{g,n}}(\td{\gm}^{m\Z})=\td{\gm}^\Z$ and
$Z_{\Pi_{g,n}}(\td{\gm}'^{m\Z})=\td{\gm}'^\Z$. Therefore, we have the identities:
$$\td{\gm}^{\Z_p}\cap\td{\gm}'^{\Z_p}=\{1\},\hspace{0.5cm}
\td{\gm}^{\Z_p}\cap\Pi_{g,n}=\td{\gm}^\Z\hspace{0.5cm} \mbox{and}\hspace{0.5cm}
\td{\gm}'^{\Z_p}\cap\Pi_{g,n}=\td{\gm}'^\Z.\hspace{0.5cm}(\ast)$$

Thus, we can apply Theorem~\ref{intersection} and conclude that $f=\td{\gm}^u\td{\gm}'^{-v}$, for some
$u,v\in\Z$. But now, the identity $\td{\gm}^{\Z_p}\cap\td{\gm}'^{\Z_p}=\{1\}$ implies that $f$ admits
a unique expression $f=\td{\gm}^s\td{\gm}'^{-t}$, with $s,t\in\Z_p$. It follows that $s=u,t=v\in\Z$. Thus,
the leaves $\ell\times\td{\gm}^s$ and $\ell'\times\td{\gm}'^t$ are in the $\Pi_{g,n}$-orbit of the the
leaves $\ell\times 1$ and $\ell'\times 1$, respectively, and their image in $\mS_{g,n}^{(p)}$ is contained in
the preferred leaf $\D\times 1$, where their intersection is identified with $\ell\cap\ell'$.
\end{proof}

\begin{proof}[Proof of Theorem~\ref{geo int}.]Before we start with the proof, let us make the following
obvious remark. For a given closed curve $\gm$ on $S_{g,n}$, let $\td{\gm}\in\Pi_{g,n}$ be an
element of the fundamental group whose free homotopy class contains $\gm$. Let $K$ be a
characteristic finite index subgroup of $\Pi_{g,n}$ and let $k>0$, be the smallest integer such that
$\td{\gm}^k\in K$. Then, for a power $\gm^s$ of $\gm$ such that $s$ divides $k$, the pull-backs of
the closed curves $\gm$ and $\gm^s$ to $S_K$ have the same set of components.
More generally, if $\gm^s$ is a power of $\gm$ and $m$ is the g.c.d. of $s$ and $k$, then each
elevation of $\gm^s$ to $S_K$ is an $s/m$-power of one of $\gm$.

Two closed curves on a Riemann surface have non-trivial geometric or algebraic
intersection if and only if any given powers of them have the same property.
By the above remarks, it is then enough to prove Theorem~\ref{geo int} when both $\gm$ and
$\gm'$ are non-power closed curves. If one of the two curves is peripheral, then either $\gm$ is
homotopic to $\gm'$ or the two curves have disjoint representatives in their homotopy classes.
In both cases, the conclusion of the theorem holds trivially.
Let us then assume that neither of the two curves is peripheral and that both closed curves
$\gm$ and $\gm'$ are geodesics for the given metric on $S_{g,n}$. In this way, the geometric
intersection number $|\gm\cap\gm'|_G=0$ if and only if either $\gm$ and $\gm'$ are disjoint or
$\gm=\gm'$ is a s.c.c.. In all these cases, the conclusion of the theorem holds for any covering
$S\ra S_{g,n}$ and there is nothing to prove.

Let us consider the case when $\gm=\gm'$ and $|\gm\cap\gm|_G\neq 0$. By replacing $S_{g,n}$
with a suitable $\cC$-covering $\td{S}\to S_{g,n}$ and $\gm=\gm'$ with two distinct elevations $\td{\gm}$ and
$\td{\gm}'$, such that $|\td{\gm}\cap\td{\gm}'|_G\neq 0$, we are then reduced to consider only the case
when the two given curves $\gm$ and $\gm'$ are a pair of distinct non-power, non-peripheral, closed curves with
non-trivial geometric intersection. Let us assume, moreover, as above that the given curves are geodesic for
the fixed complete metric on $S_{g,n}$.

There are lifts $\ell$ and $\ell'$ of $\gm$ and $\gm'$, respectively, to $\D$
such that $\ell\neq\ell'$ and $\ell\cap\ell'\neq\emptyset$.
These are geodesics in $\D$ and are the axes of two elements $\td{\gm}$ and
$\td{\gm}'$ of $\Pi_{g,n}$ whose free homotopy classes contain, respectively, $\gm$ and $\gm'$.
Then, the axes $\ell$ and $\ell'$ intersect in a single point $P$. Let us denote
by $\ol{P}$ the image of $P$ in the surface $S_{g,n}$.

For a given finite index normal subgroup $K$ of $\Pi_{g,n}$, let us denote by $\gm_K$,
$\gm_K'$ and $P_K$, respectively, the images of $\ell$, $\ell'$ and $P$ in the surface $S_K$.

Let us identify, as in the proof of Theorem~\ref{closed curve}, the inverse limits
$\ilim_{\sst K\lhd_\cC\Pi_{g,n}}\gm_K$ and $\ilim_{\sst K\lhd_\cC\Pi_{g,n}}\gm'_K$ with the closures
$\ol{\ell}$ and $\ol{\ell}'$ of $\ell$ and $\ell'$,
respectively, in the pro-$\cC$ hyperbolic solenoid. By Theorem~\ref{intersect}, there holds:
$$\ol{\ell}\cap\ol{\ell}'=\varprojlim_{\sst K\lhd_\cC\Pi_{g,n}}\gm_K\bigcap
\varprojlim_{\sst K\lhd_\cC\Pi_{g,n}}\gm'_K=\varprojlim_{\sst K\lhd_\cC\Pi_{g,n}}(\gm_K\cap\gm'_K)=P.$$
As above, here, we are using the fact that inverse limits commute with inverse limits and, in particular, 
with intersections.

By Theorem~\ref{closed curve} and the above arguments, there is then a $\cC$-open subgroup $L$ of $\Pi_{g,n}$ 
such that the curves $\gm_L$ and $\gm_L'$ are simple and all their intersection points lie below the point $P$ and, 
in particular, above the point $\ol{P}$. This implies that the intersection indices at these points are all equal.

Therefore, if $\ol{\gm}_L$ and $\ol{\gm}_L'$ are
cycles of the homology group  $H_1(\ol{S}_L,\Z)$ supported on $\gm_L$ and $\gm_L'$,
respectively, there holds $\langle\ol{\gm}_L,\ol{\gm}_L'\rangle_L=k$, for some integer $k\neq 0$,
and the same holds for every $\cC$-open subgroup $K$ of $\Pi_{g,n}$ contained in $L$.
\end{proof}

\section[Algebraic characterization of closed curves]
{An algebraic characterization of simple closed curves via a pro-$\cC$ Reidemeister pairing}\label{group theoretic}

Given a standard presentation $\Pi_g=\langle\alpha_1, \dots \alpha_g,\beta_1,\dots,\beta_g|\,
\prod_{i=1}^g[\alpha_i,\beta_i] \rangle$ of the fundamental group of a closed oriented surface $S_g$,
there are various algorithms which permit to determine whether there is a s.c.c.\ in the free homotopy
class of an element of $\Pi_g$ (for a survey on this subject, see for instance \S 3 of \cite{C-G-K-Z}).

In this section, we show that there is a pairing on the fundamental group $\Pi_g$ with values in the
$p$-adic group ring of the pro-$p$ completion (more generally, in the pro-$\cC$ completion) of $\Pi_g$
whose singular elements are precisely the loops which contain a s.c.c.\ in their free homotopy class. 
This pairing is entirely determined by the group structure of $\Pi_g$.

Let $K$ be a finite index subgroup of $\Pi_g$. The cup product on the first homology group $H_1(K)$ 
can be recovered from the group structure of $K$. Indeed, by the description of the lower central series
of one-relator groups given in \cite{Labute}, there is a short exact sequence:
$$0\ra\Z\sr{\phi}{\ra}K/[K,K]\wedge K/[K,K]\sr{\psi}{\ra}[K,K]/[[K,K]K]\ra 1,$$
where the epimorphism $\psi$ is induced by the assignment
$\alpha{\scriptstyle{\wedge}}\beta\mapsto[\alpha,\beta]$, for all $\alpha,\beta\in K$, and
the monomorphism $\phi$ is defined by $\phi(1)=\sum_{i=1}^{g_K}a_i{\scriptstyle{\wedge}}b_i$,
where $g_K$ is the genus of $S_K$ and $\{a_i,b_i\}_{i=1}^g$ is any set of generators for the
homology group $H_1(S_K,\Z)=K/[K,K]$ for which the identity $\langle a_i,b_j\rangle=\delta_{ij}$,
with $\delta_{ij}$ Kronecker's delta, holds.

Since the integral homology of a closed surface group is torsion free, there is a natural isomorphism
$H^1(K):=H^1(K,\Z)\cong\hom(H_1(K),\Z)$. An easy computation then shows that the cap product on the 
integral cohomology group is the integral dual of $\phi$:
$$\phi^\ast\co H^1(K)\wedge H^1(K)\ra\Z.$$
There is also a canonical isomorphism $H_1(K)\cong\hom(H^1(K),\Z)$. After identifying $H_1(K)$ with
$H^1(K)^\ast:=\hom(H^1(K),\Z)$ by means of the latter isomorphism, the non-degenerate skew-symmetric form 
$\phi^\ast$ on $H^1(K)$ establishes a canonical isomorphism between $H^1(K)$ and $H_1(K)$, defined by 
$x\mapsto x^\ast:=\phi^\ast(x\wedge\_)\in H^1(K)^\ast\equiv H_1(K)$, for $x\in H^1(K)$.
The cup product on $H_1(K)$ is then described by the formula:
$$\langle x^\ast,y^\ast\rangle_K=\phi^\ast(x\wedge y),\hspace{1cm}\mbox{for all }x^\ast,y^\ast\in H_1(K).$$
In this way, we get a description of the cup product on $H_1(K)$ only in terms of the group structure of $K$.

For $K$ a finite index normal subgroup of $\Pi_{g}$, let as usual $G_K:=\Pi_{g}/K$ and let $\Z[G_K]$
be the integral group ring of $G_K$. \emph{The Reidemeister pairing} (cf.\ Section~3 in \cite{Hempel2})
$$\Phi_K\co H_1(K)\times H_1(K)\to\Z[G_K]$$
is defined by $\Phi_K(a,b)=\sum_{h\in G_K}\langle a,h\cdot b\rangle_K h$, for $a,b\in H_1(K)$. 

There is a natural involution $\tau_K\co\Z[G_K]\to\Z[G_K]$, defined extending linearly the map $h\mapsto h^{-1}$
on $G_K$. With respect to the involution $\tau$, the $\Z[G_K]$-valued form $\Phi_K(\_,\_)$ is 
sesquilinear, skew-hermitian and non-degenerate.

By Lemma~3.1 in \cite{Hempel2}, if $K'$ is a finite index normal subgroup of $\Pi_{g}$ contained in $K$, 
there is a commutative diagram of Reidemeister pairings:
$$\begin{array}{clcc}
H_1(K')\!\!\!\!\!&\times H_1(K')&\sr{\Phi_{K'}}{\lra}&\Z[G_{K'}]\\
&\da{\st q_\ast\times q_\ast}&&\,\,\da{\st q_\ast}\\
H_1(K)\!\!\!\!\!&\times H_1(K)&\sr{\Phi_{K}}{\lra}&\Z[G_{K}],
\end{array}$$
where $q_\ast\co H_1(K')\to H_1(K)$ and $q_\ast\co\Z[G_{K'}]\to\Z[G_K]$ are the natural maps.

Let us switch to $\Q_p$ coefficients, so that the natural map $q_\ast\co H_1(K',\Q_p)\to H_1(K,\Q_p)$ is surjective.
For a class of finite groups $\cC$, let us then define:
$$H_1^\cC(\Q_p):=\varprojlim_{\sst K\lhd_{\cC}\Pi_{g}}H_1(K,\Q_p)\hspace{1cm}\mbox{ and }\hspace{1cm}
\Q_p[[\hP_g^\cC]]:=\varprojlim_{\sst K\lhd_{\cC}\Pi_{g}}\Q_p[G_K].$$
The set of involutions $\{\tau_K\}_{K\lhd_\cC\Pi_g}$ also forms an inverse system and its inverse limit is an involution
$\hat{\tau}^\cC\co\Q_p[[\hP_g^\cC]]\to\Q_p[[\hP_g^\cC]]$, which, for $h\in\hP_g^\cC$, is just given by $h\mapsto h^{-1}$.

Taking the inverse limit $\Phi_\cC:=\varprojlim_{\sst K\lhd_{\cC}\Pi_{g}}(\Phi_K\otimes\Q_p)$, then, with respect to 
$\hat{\tau}^\cC$, we get a $\Q_p[[\hP_g^\cC]]$-valued non-degenerate, sesquilinear, skew-hermitian form on $H_1^\cC(\Q_p)$:
$$\Phi_\cC\co H_1^\cC(\Q_p)\times H_1^\cC(\Q_p)\to\Q_p[[\hP_g^\cC]].$$

Let us now show how $\Phi_\cC$ induces a pairing on the pro-$\cC$ surface group $\hP_g^\cC$. Let us denote by
$\wh{K}$ an open normal subgroup of $\hP_g^\cC$ and let $K:=i^{-1}(\wh{K})$, where $i\co\Pi_g\hookra\hP_g^\cC$ is
the natural embedding. For an element $\alpha\in\hP_g^\cC$, let then $\nu_K(\alpha)$ be the minimal natural number such 
that there holds $\alpha^{\nu_K(\alpha)}\in\wh{K}$. Let us define a continuous map
$${\mathfrak h}_\cC\co\hP_g^\cC\to H_1^\cC(\Q_p)$$
as follows. Let $\{\wh{K}_i\}_{i\in\N}$ be a countable nest of open normal subgroups of $\hP_g^\cC$ forming a base of 
neighborhoods of the identity and let ${\mathfrak h}_i\co\wh{K}_i\to H_1(K_i,\Q_p)$ be the natural homomorphism, for $i\in\N$.
Then, let us define, for $\alpha\in\hP_g^\cC$:
$${\mathfrak h}_\cC(\alpha):=\left(\ldots,\dfrac{1}{\nu_{K_i}(\alpha)}{\mathfrak h}_i(\alpha^{\nu_{K_i}(\alpha)}),\ldots,
\dfrac{1}{\nu_{K_1}(\alpha)}{\mathfrak h}_1(\alpha^{\nu_{K_1}(\alpha)})\right)\in H_1^\cC(\Q_p).$$
In Section~\ref{separating}, we will show that the restriction of the map ${\mathfrak h}_\cC$ to $\Pi_g$ is injective 
(cf.\ Remark~\ref{separation3}). Composing the map ${\mathfrak h}_\cC$ with the pairing $\Phi_\cC$, we get a continuous pairing
$${\mathfrak R}_\cC\co\hP_g^\cC\times\hP_g^\cC\to\Q_p[[\hP_g^\cC]],$$
which we call \emph{the pro-$\cC$ Reidemeister pairing on} $\hP_g^\cC$. By restriction, 
we get as well a pairing on the discrete surface group $\Pi_g$. By the definition, it is clear that
the value of ${\mathfrak R}_\cC(\alpha,\beta)$ only depends on the conjugacy classes of $\alpha$ and $\beta$ in 
$\hP_g^\cC$. Therefore, if we denote by $\cL^{or}$ the set of homotopy classes of oriented closed curves on $S_g$,
we get also a well defined pro-$\cC$ Reidemeister pairing:
$${\mathfrak R}_\cC\co\cL^{or}\times\cL^{or}\to\Q_p[[\hP_g^\cC]].$$

\begin{proposition}\label{character}There is a character $\chi_p\co\aut(\hP_g^\cC)\to\Z_p^\ast$ such that, for all $f\in\aut(\hP_g^\cC)$ 
and $\alpha,\beta\in\hP_g^\cC$, there holds  ${\mathfrak R}_\cC(f(\alpha),f(\beta))=\chi_p(f)\cdot{\mathfrak R}_\cC(\alpha,\beta)$.
For $f\in\inn(\hP_g^\cC)$, there holds $\chi_p(f)=1$. Thus, the character $\chi_p$ descends to a character
$\bar\chi_p\co\out(\hP_g^\cC)\to\Z_p^\ast$.
\end{proposition}

\begin{proof}For $\wh{K}$ an open characteristic subgroup of $\hP_g^\cC$, 
let us consider the effect of the automorphism $f$ on the short exact sequence
$$0\ra\Z_p\sr{\phi}{\ra}(\wh{K}/[\wh{K},\wh{K}]\wedge \wh{K}/[\wh{K},\wh{K}])\otimes\Z_p\sr{\psi}{\ra}
([\wh{K},\wh{K}]/[[\wh{K},\wh{K}]\wh{K}])\otimes\Z_p\ra 1.$$
There holds $f_\ast\circ\phi(1)=\sum_{i=1}^{g_K}f_\ast(a_i){\scriptstyle{\wedge}}f_\ast(b_i)=
\chi_p(f)\sum_{i=1}^{g_K}a_i{\scriptstyle{\wedge}}b_i$, for some $\chi_p(f)\in\Z_p^\ast$,
and thus $f_\ast\circ\phi=\chi_p(f)\cdot\phi$. By double duality, there then holds 
$(\Phi_K\otimes\Q_p)\circ f_\ast=\chi_p(f)\cdot(\Phi_K\otimes\Q_p)$. Since open characteristic subgroups
form a base of neighborhoods of the identity of $\hP_g^\cC$, passage to the inverse limit yields the claim of the proposition.

As observed above, the value of ${\mathfrak R}_\cC(\alpha,\beta)$ only depends on the conjugacy
classes of $\alpha$ and $\beta$ in $\hP_g^\cC$. In particular, the pairing is invariant under inner automorphisms
of $\hP_g^\cC$ and so the second part of the proposition follows as well.
\end{proof}

\begin{remark}\label{Galois}If we identify the profinite completion $\wh{\Pi}_g$ of $\Pi_g$ with the \'etale fundamental 
group of $C\times_\Q\Spec(\ol{\Q})$, where $C$ is a projective smooth curve of genus $g$ defined over $\Q$, there is a 
natural faithful representation $\rho_C\co G_\Q\hookra\out(\hP_g)$, where $G_\Q$ is the absolute Galois group 
(cf.\ (ii) Theorem~C in \cite{HM} and Theorem~7.7 in \cite{Boggi}). It is then not difficult to see that the character 
$\bar\chi_p$ restricts on $G_\Q$ to the $p$-adic cyclotomic character. In particular, this implies that both characters 
$\chi_p$ and $\bar\chi_p$ are surjective.
\end{remark}
  
Let ${\mathfrak h}\co K\to H_1(K)$ be the natural map. From the definition of the standard Reidemeister pairing 
$\Phi_K$ given above, it is then clear that, for $\alpha,\beta\in\Pi_g$, the submodules $V_\alpha^K$ and $V_\beta^K$ 
are reciprocally orthogonal if and only if $\Phi_K({\mathfrak h}(\alpha^{\nu_K(\alpha)})),{\mathfrak h}(\beta^{\nu_K(\beta)}))=0$. 
Similarly, $V_\gm^K$ is totally isotropic if and only if $\Phi_K({\mathfrak h}(\gm^{\nu_K(\gm)}),{\mathfrak h}(\gm^{\nu_K(\gm)}))=0$. 
Combining this remark with Theorem~\ref{geo int} and its corollaries, we get the following characterization of simple 
closed curves on $S_g$. For a closed curve $\gm$ on $S_g$, let us denote by $\td{\gm}\in\Pi_g$ an element whose 
free homotopy class contains $\gm$ and by $\vec{\gm}\in\cL^{or}$ the homotopy class of $\gm$ with a fixed orientation:

\begin{theorem}\label{reidemeister}Let $\cC$ be a class of finite groups and $g\geq 2$.
\begin{enumerate}
\item A pair of closed curves $\alpha$ and $\beta$ on the closed Riemann surface $S_g$ have trivial geometric intersection
if and only if there holds ${\mathfrak R}_\cC(\vec{\alpha},\vec{\beta})=0$.
\item The homotopy class of a non-power closed curve $\gm$ on $S_g$ contains a
simple closed curve if and only if there holds ${\mathfrak R}_\cC(\vec{\gm},\vec{\gm})=0$.
\item Let $\alpha$ and $\beta$ be non-homotopic s.c.c.'s on $S_g$.
Then $\alpha$ and $\beta$ are homotopic to s.c.c.'s which meet transversally in a single point
if and only, for some $\td{\alpha},\td{\beta}\in\Pi_g$, there holds 
${\mathfrak R}_\cC([\td{\alpha},\td{\beta}],[\td{\alpha},\td{\beta}])=0$.
\end{enumerate}
\end{theorem}

A consequence of Theorem~\ref{reidemeister} is that the topology of the closed Riemann surface
$S_g$ is entirely determined by the natural embedding $\Pi_g\hookra\hP_g^\cC$ and by the algebraic structure
of $\hP_g^\cC$, for any class $\cC$ of finite groups.

\section{Distinguishing closed curves}\label{separating}
As a first consequence of Theorem~\ref{geo int}, we have the following result:

\begin{theorem}\label{separation}Let $\gm$ and $\gm'$ be closed curves not homotopic to powers of
peripheral s.c.c.'s on the Riemann surface $S_{g,n}$ and let $p\in\N$ be a fixed prime. Then, the two curves
are non-homotopic, if and only if, there is a normal, unramified $p$-covering $p_K\co S_K\ra S_{g,n}$
such that the $G_K$-invariant submodules $V^K_\gm$ and $V^K_{\gm'}$ of $H_1(\ol{S}_K)$ are distinct.
\end{theorem}

\begin{proof}One implication is obvious. So let us suppose that, for all normal, finite, unramified
$p$-coverings $p_K\co S_K\ra S_{g,n}$, there holds $V^K_\gm=V^K_{\gm'}$. In particular,
there holds $\langle x,y\rangle_K=0$, for all $x\in V^K_\gm$, $y\in V^K_{\gm'}$ and all
$p$-open normal subgroups $K$ of $\Pi_{g,n}$.

From Theorem~\ref{geo int}, it follows that $\gm$ and $\gm'$ have trivial geometric intersection.
But then, for powers of non-peripheral s.c.c.'s, with trivial geometric intersection, the above condition
implies that $\gm$ and $\gm'$ are homotopic to some powers of the same s.c.c.\ $\td{\gm}$.
Let us then assume that the closed curve $\gm$ is homotopic to a power $\td{\gm}^s$ of $\td{\gm}$
and the closed curve $\gm'$ is homotopic to a power $\td{\gm}^t$ of $\td{\gm}$.

If $\td{\gm}$ is homologically non-trivial on the closed surface $S_g=\ol{S}_{g,n}$, the
homology classes of $\gm$ and $\gm'$ generate the same submodule of $H_1(S_g)$ only if $s=t$
and the curves $\gm$ and $\gm'$ are then homotopic.

If instead $\td{\gm}$ is homologically trivial on the closure $S_g$ of $S_{g,n}$, the s.c.c.\ $\td{\gm}$ is
separating and bounds on each side a subsurface of negative Euler characteristic (because $\td{\gm}$
is non-peripheral). It is then not difficult to see (cf.\ the proof of Lemma~3.10 \cite{sym}) that there is an
abelian, unramified, $p$-covering $p_K\co S_K\ra S_{g,n}$ such that all connected components of
$p_K^{-1}(\td{\gm})$ are non-separating and are mapped bijectively onto $\td{\gm}$ by $p_K$.
In particular, they have non-trivial integral homology classes in $H_1(\ol{S}_K)$. As above,
$V^K_\gm=V^K_{\gm'}$ then implies that $s=t$ and that $\gm$ and $\gm'$ are homotopic.
\end{proof}

\begin{remark}\label{peripheral}In terms of the homology of $p$-coverings, peripheral closed curves 
on $S_{g,n}$ are characterized by the property that, for some fixed $p$ and every $p$-open normal 
subgroup $K$, there holds $V^K_\gm=\{0\}$. The necessity of the condition is obvious. That it is also 
sufficient immediately follows from Theorem~\ref{geo int} and the fact that having trivial geometric 
intersection with any closed curve on $S_{g,n}$ characterizes peripheral closed curves.
In particular, in terms of the homology of normal $p$-coverings it is possible to distinguish a peripheral
curve from a non-peripheral curve but not a peripheral curve from another.
\end{remark}

For arbitrary closed curves, we then derive the following weaker separation property:

\begin{theorem}\label{separation2}Let $\alpha$ and $\beta$ be non-peripheral closed curves on
$S_{g,n}$. Then, the two curves are non-homotopic if and only if, for any fixed prime $p$, there is
a normal, unramified $p$-covering $p_K\co S_K\ra S_{g,n}$ such that every cycle of $H_1(\ol{S}_K)$
supported on an elevation of $\alpha$ to $S_K$ is distinct from every cycle supported
on an elevation of $\beta$ to $S_K$.
\end{theorem}

\begin{proof}The case when a power of one curve is homotopic to a power of the other can be
treated as in the proof of Theorem~\ref{separation}. Let us then assume that no power of one
curve is homotopic to a power of the other.

By Theorem~\ref{closed curve}, there is a $p$-open
normal subgroup $L$ of $\Pi_{g,n}$ such that each component of
the pull-backs of $\alpha$ and $\beta$ to $S_L$ is homotopic to a power of a simple closed
curve. Since no power of $\alpha$ is homotopic to a power of $\beta$, all components of
the pull-backs of $\alpha$ and $\beta$ to $S_L$ are pairwise non-homotopic.

At this point, we can apply Theorem~\ref{separation} and conclude that there is a characteristic,
unramified $p$-covering $p_{KL}\co S_K\ra S_L$ such that, if $\gm_\alpha$ and
$\gm_\beta$ are, respectively, elevations of $\alpha$ and $\beta$ to $S_L$,
the submodules $V_{K,L,\gm_\alpha}$ and $V_{K,L,\gm_\beta}$ of $H_1(\ol{S}_K)$,
generated by the cycles supported, respectively, on the elevations of
$\gm_\alpha$ and $\gm_\beta$ from $S_L$ to $S_K$, are distinct.

Let $G_{K,L}$ be the Galois group of the covering $p_{KL}\co S_K\ra S_L$.
If some cycle $\td{\alpha}$ of $H_1(\ol{S}_K)$, supported on an elevation of
$\alpha$ to $S_K$, were in the same homology class of a cycle $\td{\beta}$, supported
on an elevation of $\beta$ to $S_K$, this would imply that the $G_{K,L}$-orbit
of $\td{\alpha}$ generates the same submodule of $H_1(\ol{S}_K)$ as the $G_{K,L}$-orbit of
$\td{\beta}$, in contrast with the above conclusion.
\end{proof}

\begin{remark}\label{separation3}By Lemma~3.10 \cite{sym}, if the subgroup $K$ of Theorem~\ref{separation2} 
is contained in $[\Pi_g,\Pi_g]\Pi_g^\ell$, for some integer $\ell\geq 2$, then, for $\gm$ a s.c.c.\ on $S_g$, 
the inverse image $p_K^{-1}(\gm)$ does not contain separating curves or cut pairs. So that two cycles of 
$H_1(S_K)$ supported on distinct connected component of $p_K^{-1}(\gm)$ are also distinct. Together with 
Theorem~\ref{separation2}, this implies that the map ${\mathfrak h}_\cC\co\Pi_g\to H_1^\cC(\Q_p)$, 
defined in Section~\ref{group theoretic}, is injective.
\end{remark}

Let $p$ be a prime. A group $G$ is conjugacy $p$-separable if, whenever $x$ and $y$ are
non-conjugate elements of $G$, there exists some finite $p$-quotient of G in which the images of $x$
and $y$ are non-conjugate.

An almost immediate consequence of Theorem~\ref{separation2} is conjugacy $p$-separability
of fundamental groups of oriented Riemann surface. This was well known for open surfaces but,
in the closed surface case, it was proved only recently by Paris \cite{Paris}.

\begin{theorem}\label{p-separable}The fundamental group $\Pi_{g,n}$ of an oriented surface is
conjugacy $p$-separable.
\end{theorem}

\begin{proof}If $\Pi_{g,n}$ is abelian, the result is trivial. So, let us assume $2g-2+n>0$ and
let be given $\alpha,\beta\in\Pi_{g,n}$ belonging to distinct conjugacy classes.

Let us consider first the case when both $\alpha$ and $\beta$ contain peripheral curves
in their homotopy classes bounding, respectively, the punctures $P_{i_1}$ and $P_{i_2}$, with
$1\leq i_1,i_2\leq n$. Since their conjugacy classes are distinct, either $i_1\neq i_2$ and
$n\geq 2$, or $i_1=i_2$ and $\alpha^s$ is conjugated to $\beta^t$, for some $s\neq t\in\N$.

Let $N$ be the kernel of the natural epimorphism $\Pi_{g,n}\ra\Pi_{g}$ induced by filling in the
punctures of $S_{g,n}$ and let $L:=N[\Pi_{g,n},\Pi_{g,n}]\Pi^p_{g,n}$, then there holds
$\alpha,\beta\in L$. If $n\geq 1$, the surface $S_L$ associated to $L$ has at least two punctures,
otherwise, if $n\geq 2$, the surface $S_L$ has at least three punctures. In any case, the loops
$\alpha$ and $\beta$ lift to loops on $S_L$ which define distinct non-trivial homology
classes $\bar\alpha$ and $\bar\beta$ in $H_1(S_L,\Z/p^s)=H_1(L,\Z/p^s)$, for some $s\in\N^+$.
An element of $\Pi_{g,n}$ in the conjugacy class of $\alpha$ lifts in $S_L$ to a peripheral loop bounding
a puncture lying above $P_{i_1}$ and determines a cycle of $H_1(S_L,\Z/p^s)$
in the $G_L$-orbit of $\bar\alpha$. Moreover, all elements in the $G_L$-orbit of $\bar\alpha$ are
of this type.

It is then clear that, whether or not $i_1\neq i_2$, the images of $\alpha$ and $\beta$ in the homology
group $H_1(L,\Z/p^s)$ are in distinct $G_L$-orbits. Therefore, also their images in the finite $p$-quotient
$\Pi_{g,n}/[L,L]L^{p^s}$, which contains  $L/[L,L]L^{p^s}\cong H_1(L,\Z/p^s)$ as a subgroup, are
non-conjugate.

Let us then consider the case in which one of the two elements, say $\alpha$, contains a peripheral
curve in its homotopy class and $\beta\in [\Pi_{g,n},\Pi_{g,n}]$.
Let $L$ be defined as above. Then, there holds as well $\alpha,\beta\in L$ and their images
$\bar\alpha$ and $\bar\beta$ in $H_1(L,\Z/p)$ are in distinct $G_L$-orbits. As above, it follows
that the conjugacy classes of $\alpha$ and $\beta$ are separated by the finite $p$-quotient
$\Pi_{g,n}/[L,L]L^p$.

The case in which one of the two elements, say $\alpha$, contains a peripheral curve in its homotopy
class and $\beta$ instead a non-separating curve is straightforward.

Let us then assume that neither of the two elements contain a peripheral curve in its
homotopy class. By Theorem~\ref{separation2},
there is a $p$-open characteristic subgroup $K$ of $\Pi_{g,n}$
and a positive integer $m$ such that the $G_K$-orbits of $\tilde{\alpha}$ and $\tilde{\beta}$ in
$H_1(K,\Z/p^m)$ are distinct, where $\tilde{\alpha}$ and $\tilde{\beta}$ denote the images
in $H_1(K,\Z/p^m)$ of the minimal positive powers $\alpha^s$ and $\beta^t$ contained in $K$.
We can assume that $s=t$, otherwise it is already clear that the images of $\alpha$ and $\beta$
in the $p$-quotient $\Pi_{g,n}/K$ are not in the same conjugacy class.

The subgroup $[K,K]K^{p^m}$ of $K$ generated by commutators and $p^m$ powers is a $p$-open
characteristic subgroup of $\Pi_{g,n}$, such that the image of $K$ in
the $p$-group $\Pi_{g,n}/[K,K]K^{p^m}$ is naturally isomorphic to the homology group
$H_1(K,\Z/p^m)$. Therefore, the images $\tilde{\alpha}$ and $\tilde{\beta}$ of $\alpha^s$ and
$\beta^s$ in $\Pi_{g,n}/[K,K]K^{p^m}$ belong to distinct conjugacy classes.
The same then is true for the images there of $\alpha$ and $\beta$.
\end{proof}

\appendix

\section{Appendix on pro-$p$ surface groups}
In this appendix, we prove two basic results on  pro-$p$ surface groups which
are used in the paper. Let us then denote, as above, by $\Pi:=\pi_1(S,\ast)$ the fundamental group
of an oriented hyperbolic Riemann surface $S$ of finite type and by $\hP^{(p)}$ its pro-$p$ completion.

A first basic property of the pro-$p$ completion of a surface group is given in the following theorem.
Although this is well known, we prefer to adapt here to the pro-$p$ case the simple geometric proof by Hempel
(cf.\ \cite{Hempel}) that surface groups are residually finite, which is more in the spirit of the present paper:

\begin{theorem}\label{injective}The canonical homomorphism $\Pi\ra\hP^{(p)}$ is injective.
Otherwise stated, oriented surface groups are residually $p$-groups for all primes $p\geq 2$.
\end{theorem}

\begin{proof}It is enough to show that, for a given $1\neq\alpha\in\Pi$ and a map
$f\co (S^1,\ast)\ra (S,\ast)$ representing $\alpha$, there is a characteristic unramified $p$-covering
$\pi\co S'\ra S$ such that $f$ does not lift to a map $\td{f}\co (S^1,\ast)\ra (S',\ast)$, where
we denote also by $\ast$ a base point on $S'$ lying over the base point on $S$ for $\Pi$.
It is not restrictive to assume that $f$ is an immersion whose image has transversal
self-intersection points.

Let us proceed by induction on the cardinality of the singular set $s(f)$ of the map $f$.

If $f$ is an embedding, there is a characteristic unramified $p$-covering $\pi\co S'\ra S$ such that
the connected components of $\pi^{-1}(f(S^1))$ are non-separating simple closed curves
(cf.\ Lemma~3.10 \cite{sym}). Let $S''\ra S'$ be the unramified $p$-covering with covering
transformation group isomorphic to the homology $H_1(S',\Z/p)$. Then, the composition
$\pi'\co S''\ra S$ is a characteristic unramified $p$-covering with the desired property.

If $s(f)\neq\emptyset$, let $U$ be a regular neighborhood of $f(S^1)$ in $S$. There is a simple
loop $g\co (S^1,\ast)\ra (U,\ast)$ which represents a non-trivial element of $\Pi$.

By the preceding case, there is a characteristic unramified $p$-covering $\pi\co S'\ra S$ such that
$g$ does not lift to a map $\td{g}\co (S^1,\ast)\ra (S',\ast)$. If $f$ does not lift too to $S'$, we are done.
If $f$ does lift to a map $\td{f}\co (S^1,\ast)\ra (S',\ast)$, then $s(\td{f})\subseteq s(f)$.
If $s(\td{f})=s(f)$, then $\pi|_{\td{f}(S^1)}$ would be an embedding, and $\pi$ would map
a neighborhood of $\td{f}(S^1)$ homeomorphically onto $U$. But then $g$ would lift to $S'$.
Therefore, there holds $s(\td{f})\subsetneq s(f)$ and the proof follows by induction.
\end{proof}

By Theorem~\ref{injective}, we can and do identify the group $\Pi$ with its image in the pro-$p$
completion $\hP^{(p)}$. The second result we need is the following:

\begin{theorem}\label{intersection}Let $\delta_1,\delta_2\in\Pi$ be elements such that the cyclic
subgroups $\delta_1^{\Z}$ and $\delta_2^{\Z}$ are $p$-closed in $\Pi$
and let $\delta_1^{\Z_p}$ and $\delta_2^{\Z_p}$ be the pro-cyclic subgroups topologically
generated by $\delta_1$ and $\delta_2$ in the pro-$p$ completion $\hP^{(p)}$. Then, there holds:
$$\Pi\cap\delta_1^{\Z_p}\cdot\delta_2^{\Z_p}=\delta_1^\Z\cdot\delta_2^\Z.$$
\end{theorem}

\begin{proof}If $\delta_1=\delta_2^{\pm 1}$, the statement is just a different formulation of the
fact that $\delta_1^{\Z}$ and $\delta_2^{\Z}$ are $p$-closed in $\Pi$. Let us then assume
$\delta_1\neq\delta_2^{\pm 1}$ which then implies that $\delta_1^{\Z_p}\cap\delta_2^{\Z_p}=\{1\}$.

If $\Pi$ is a free group, the equality in the statement of the theorem is a particular case of 
Theorem~6.1 \cite{RZ2}. However, a more direct  and self-contained proof immediately 
follows from the lemma:

\begin{lemma}\label{rz} Let $H_1$ and $H_2$ be finitely generated subgroups of a 
free group $F$ which are $p$-closed in $F$ and such that $H_1\cap H_2=\{1\}$.
Then the set $H_1\cdot H_2$ is $p$-closed in $F$.
\end{lemma}

\begin{proof} The proof is an adaptation to the pro-$p$ topology of Niblo's proof of
Theorem~3.2 \cite{Niblo}. By Theorem 5.7 \cite{RZ2}, there is a $p$-open subgroup $U$
of $F$ such that $H_1$ is a free factor of $U$. Since $(H_1\cdot H_2)\cap U$ is $p$-closed if and only
if $H_1\cdot H_2$ is $p$-closed, we may assume that $F=K\ast H_1$ and then form the double
$F\ast_{H_1}F=K\ast H_1\ast K$ along the subgroup $H_1$.
Let $\tau_i\co F\hookrightarrow F\ast_{H_1}F$, for $i=1,2$, be the natural $p$-continuous
monomorphism which identifies $F$, respectively, with the first and the second factor. The assignment
$x\mapsto \tau_1^{-1}(x)\tau_2(x)$ then defines a $p$-continuous map $\eta\co F\ra F\ast_{H_1}F$
and, by  Lemma~3.1 \cite{Niblo}, there holds:
$$\eta^{-1}(\langle\tau_1(H_2),\tau_2(H_2)\rangle)=H_1\cdot H_2.$$

It remains to show that $\langle\tau_1(H_2),\tau_2(H_2)\rangle$ is closed in the pro-$p$ topology of
$F\ast_{H_1}F$. Let us consider the natural $p$-continuous epimorphism $\mu\co F\ast_{H_1}F\ra F$
which identifies the two factors of the double. By hypothesis, there holds $H_1\cap H_2=\{1\}$ and then
$\mu^{-1}(H_2)\cap H_1=\{1\}$. From the Kurosh Subgroup Theorem
(cf.\ Theorem~14 in \cite{serre-77}), it follows:
$$\mu^{-1}(H_2)=(\mu^{-1}(H_2)\cap F)\ast(\mu^{-1}(H_2)\cap F)\ast D=
\tau_1(H_2)\ast\tau_2(H_2)\ast D,$$
where $D$ is some free factor. Since $\mu$ is $p$-continuous, $\mu^{-1}(H_2)$ is $p$-closed in
$F\ast_{H_1}F$ and then so is its free factor
$\tau_1(H_2)\ast\tau_2(H_2)=\langle\tau_1(H_2),\tau_2(H_2)\rangle$
(cf.\ Corollary 5.6 \cite{RZ2}).
\end{proof}

Let us then consider the case when $\Pi$ is a hyperbolic closed surface group. The idea is to reduce the proof
to Lemma~\ref{rz} as done above for $\Pi$ free. In the case under consideration, at least the group 
$F=\langle\delta_1,\delta_2\rangle$ is free, because it has infinite index in $\Pi$. 
This follows from the fact that a subgroup of infinite index of
a hyperbolic oriented surface group identifies with the fundamental group of a non-compact
oriented Riemann surface, which is free. 

We cannot directly replace $\Pi$ by $F$ in the proof of the theorem, because we do not know 
whether $F$ is closed in the pro-$p$ topology of $\Pi$ and so whether its closure $\ol{F}$ in $\hP^{(p)}$ 
coincides with the pro-$p$ completion of $F$.

Let us then observe that the closure $\ol{F}$ of $F$ in the pro-$p$ completion $\hP^{(p)}$ has infinite index. 
Indeed, since $\Pi$ is a closed surface group, its pro-$p$ completion $\hP^{(p)}$ is a Demu\v{s}kin pro-$p$ group
and so are its open subgroups. Now, a $2$-generated Demu\v{s}kin pro-$p$ group is soluble, therefore $\ol F$ 
is not open in $\hP^{(p)}$ and so is of infinite index. From part (b) of Exercise~5, Ch. I, \S 4.5,
\cite{serre}, it then follows that $\ol{F}$ is a free pro-$p$ group. 

Now, the intersection $F'=\ol F\cap \Pi$ has also infinite index in $\Pi$ and hence is a
free group. Moreover, the subgroup $F'$ is closed in the pro-$p$ topology of $\Pi$
and so its closure $\ol F$ in $\hP^{(p)}$ coincides with its pro-$p$ completion. 
Thus, by replacing $\Pi$ with $F'$, in the required equality 
$\Pi\cap\delta_1^{\Z_p}\cdot\delta_2^{\Z_p}=\delta_1^\Z\cdot\delta_2^\Z$, 
we are finally reduced to the case when $\Pi$ is a free group which we already treated in Lemma~\ref{rz}. 
\end{proof}


\vspace{0.3cm}

\noindent Marco Boggi,\\ Departamento de Matem\'aticas, Universidad de los Andes, \\
Carrera $1^a$ $\mathrm{N}^o$ 18A-10, Bogot\'a, Colombia.
\\
E--mail:\,\,\, marco.boggi@gmail.com

\bigskip

\noindent Pavel Zalesskii,\\ Departamento da Matem\'atica, Universidade de Bras\'{\i}lia, \\
70910-900, Bras\'{\i}lia-DF, Brasil.
\\
E--mail:\,\,\, pz@mat.unb.br


\begin{thebibliography}{PPP}

 \bibitem{A-M} M. Artin, B. Mazur. \textsl{Etale Homotopy}. Springer Lecture Notes in Mathematics 
 n. {\bf 100} (1969). 

\bibitem{Boggi} M.\ Boggi. \textsl{On the procongruence completion of the Teichm\"uller
modular group}. Trans.\ Amer.\ Math.\ Soc.\ {\bf 366} (2014), no.\ 10, 5185--5221.

\bibitem{sym} M.\ Boggi. \textsl{Galois covers of moduli spaces of curves and loci of
curves with symmetries}. Geom.\ Dedicata n.\ {\bf 168} (2014), 113--142.



\bibitem{C-G}M. Chas, S. Gadgil. \textsl{The Goldman bracket determines intersection numbers for surfaces 
and orbifolds}. \texttt{arXiv:1209.0634v2} (2014).

\bibitem{C-K}M. Chas, F. Krongold. \textsl{An algebraic characterization of simple closed
curves on surfaces with boundary}. JTA vol. 2, issue {\bf 3} (2010), 395--417.

\bibitem{C-K2}M. Chas, F. Krongold. \textsl{Algebraic characterization of simple closed curves
via Turaev's cobracket}. \texttt{arXiv:1009.2620v1} (2010).


\bibitem{C-G-K-Z}D.J. Collins, R.I. Grigorchuk, P.F. Kurchanov, H. Zieschang.
\textsl{Combinatorial group theory and Applications to Geometry}. Springer Verlag (1998).



\bibitem{FH} B. Farb, S. Hensel. \textsl{Extracting topology from representations in covers of finite graphs}. Preprint. 


\bibitem{Hempel} J. Hempel. \textsl{Residual finiteness of surface groups.}
Proc. Amer. Math. Soc. {\bf 144} , no. 1 (1972), 323.

\bibitem{Hempel2} J. Hempel. \textsl{One-relator surface groups.}
Math. Proc. Amer. Camb. Phil. Soc. {\bf 108}  (1990), 467--474.

\bibitem{HM}
Y.~Hoshi, S.~Mochizuki. \textsl{On the combinatorial anabelian geometry
of nodally nondegenerate outer representations.} Hiroshima Math.\ J.\ {\bf 41}, (2011), 275--342. 

\bibitem{Jaco} W. Jaco. \textsl{Heegaard splittings and splitting homomorphisms.}
Trans. Amer. Math. Soc.  {\bf 144} (1969), 365--379.

\bibitem{Labute}J.P. Labute. \textsl{On the Descending Central Series of Groups with a Single
Defining Relation.} Journal of Algebra {\bf 14}, (1970), 16--23.




\bibitem{Niblo}G.A. Niblo. \textsl{Separability Properties of Free Groups and Surface Groups.}
J. Pure Appl. Algebra, {\bf 78} (1992), 77--84.


\bibitem{Paris}L. Paris. \textsl{Residual $p$ properties of mapping class groups and
surface groups.} Trans. Amer. Math. Soc. {\bf 361} (2009), 2487--2507.



\bibitem{RZ2}L. Ribes and P.A. Zalesskii, \textsl{Pro-$p$ Trees.} In \emph{New Horizons in pro-$p$
Groups.} M. du Sautoy, D. Segal and A. Shalev editors. Progress in Mathematics {\bf 184},
Birkh\"auser, Boston (2000).


\bibitem{Scott1}P. Scott. \textsl{Subgroups of surface groups are almost geometric.}
J. London Math. Soc. (2), {\bf 17} (1978), 555--565.

\bibitem{Scott2}P. Scott. \textsl{Correction to "Subgroups of surface groups are almost
geometric".} J. London Math. Soc. (2), {\bf 32} (1985), 217--220.

\bibitem{serre-77}J.P. Serre. \textsl{Trees.} Springer Monographs in Mathematics (1980).

\bibitem{serre}J.P. Serre. \textsl{Galois Cohomology.} Springer Lecture Notes in
Mathematics {\bf 5}  (1997).


\bibitem{Sta}J.R. Stallings. \textsl{How not to prove the Poincar\'e conjecture.} In
\emph{Topology seminar, Wisconsin, 1965.} Ann. of Math. Studies {\bf 60}, Princeton University Press
(1966), 83--88.


\bibitem{Tu}V. G. Turaev. \textsl{Skein quantization of Poisson algebras of loops on surfaces}.
Ann. Sci. \'Ecole Norm. Sup. (4) {\bf 24}, No. 6 (1991), 635--704.


\end{thebibliography}
\end{document}